\documentclass[11pt]{article}

\usepackage{amsfonts}
\usepackage{shuffle,yfonts,amsmath,amssymb,amsthm}
\usepackage{hyperref}
\hypersetup{pdfstartview={XYZ null null 1.00}}
\usepackage{pdfpages}
\usepackage{xcolor}

\newcommand\nc{\newcommand}
\let\Lsh\relax
\nc{\gd}{{\delta}}
\nc{\gs}{{\sigma}}
\nc{\gth}{{\theta}}
\nc{\mi}{{\,\mathrm{i}\,{}}}
\nc{\dd}{\mathrm{d}}
\nc{\lra}{{\longrightarrow}}
\nc{\tD}{\widetilde{D}}

\DeclareMathOperator{\RR}{\mathfrak{R}}
\DeclareMathOperator{\arcsinh}{arcsinh}
\DeclareMathOperator{\Li}{Li}
\DeclareMathOperator{\Ls}{Ls}
\DeclareMathOperator{\Lsc}{Lsc}
\DeclareMathOperator{\Lsh}{Lsh}
\DeclareMathOperator{\Lshch}{Lshch}
\DeclareMathOperator{\Cl}{Cl}

\let\Re\relax
\let\Im\relax
\DeclareMathOperator{\Re}{Re}
\DeclareMathOperator{\Im}{Im}

\nc{\CMZV}{\mathsf{CMZV}}
\nc{\QZV}{\mathsf{QZV}}
\nc{\R}{{\mathbb{R}}}
\nc{\Q}{{\mathbb{Q}}}
\nc{\CP}{{\mathbb{CP}}}

\allowdisplaybreaks
\numberwithin{equation}{section}

\theoremstyle{plain}
\newtheorem{thm}{Theorem}[section]
\newtheorem{lem}[thm]{Lemma}
\newtheorem{cor}[thm]{Corollary}
\newtheorem{con}[thm]{Conjecture}

\theoremstyle{definition}
\newtheorem{defn}{Definition}[section]
\newtheorem{re}[thm]{Remark}

\usepackage{mathtools}
\let\overlineO\overline
\renewcommand{\overline}[1]{\overlineO{\mathclap{\phantom{I}}#1}}
\let\bar\overline

\begin{document}

\title{\bf On two conjectures of Sun \\ concerning Ap\'ery-like series}

\author{
{Steven Charlton${}^{a,}$\thanks{Email: steven.charlton@uni-hamburg.de, ORCID 0000-0002-2815-1885.}~,
	 Herbert Gangl${}^{b,}$\thanks{Email: herbert.gangl@durham.ac.uk, ORCID 0000-0001-7785-263X.}~,
	  Li Lai${}^{c,}$\thanks{Email: lilaimath@gmail.com, ORCID 0000-0001-8699-8753.}~,
	   Ce Xu${}^{d,}$\thanks{Email: cexu2020@ahnu.edu.cn, ORCID 0000-0002-0059-7420.}~
	    and Jianqiang Zhao${}^{e,}$\thanks{Email: zhaoj@ihes.fr, ORCID 0000-0003-1407-4230.}}\\[1mm]
\small $a$. Fachbereich Mathematik (AZ), Universit\"at Hamburg,\\ \small Bundesstra{\ss}e 55, 20146 Hamburg, Germany\\
\small $b$. Department of Mathematical Sciences, Durham University,\\ \small Durham DH1 3LE, United Kingdom\\
\small $c$. Department of Mathematical Sciences, Tsinghua University,\\ \small  Beijing 100084, P.R. China\\
\small $d$. School of Mathematics and Statistics, Anhui Normal University,\\ \small Wuhu 241002, P.R. China\\
\small $e$. Department of Mathematics, The Bishop's School, La Jolla,\\ \small CA 92037, United States of America}

\date{}
\maketitle

\noindent{\bf Abstract.} In this paper, we shall prove two conjectures of Z.-W. Sun concerning Ap\'ery-like series. One of the series is alternating whereas the other one is not. Our main strategy is to convert the series (resp.~the alternating series) to log-sine-cosine (resp.~log-sinh-cosh) integrals. Then we express all these integrals in terms of single-valued Bloch-Wigner-Ramakrishnan-Wojtkowiak-Zagier polylogarithms.
The conjectures then follow from a few highly non-trivial functional equations of the polylogarithms of weight $3$ and $4$.



\section{Introduction}

Let $\zeta(s) := \sum_{n=1}^{\infty} n^{-s}$ be the Riemann zeta function for $\Re s >1$. In the 1979's proof \cite{Ape1979} of the irrationality of $\zeta(3)$, R. Ap\'ery made use of the following infinite series involving central binomial coefficients:
\[\sum_{n=1}^{\infty} \frac{(-1)^{n-1}}{n^{3}\binom{2n}{n}}=\frac{2}{5}\zeta(3).\]
Since then, the \textit{Ap\'ery-like} series have attracted much attention.
We refer the reader to \cite{XZ2022} for a survey on recent progress.

The aim of this paper is to prove two conjectures of Z.-W. Sun concerning Ap\'ery-like series. These conjectures were published first in \cite{Sun2015} and included in Sun's book \cite{Sun2021}. Define the classical harmonic numbers
\[H_n:=\sum_{k=1}^{n}\frac{1}{k},\quad\text{for~}n=1,2,3,\ldots.\]
Let
\[\beta(s):=\sum_{n=0}^{\infty}\frac{(-1)^n}{(2n+1)^s},\quad\text{for~}\Re s>0\]
be the Dirichlet beta function.

\begin{con}[{\cite[Conjectures 10.59(i) and 10.60]{Sun2021}}] We have
\begin{align}
&\sum_{n=0}^\infty\frac{\binom{2n}{n}}{(2n+1)^3 16^n}\left(9H_{2n+1}+\frac{32}{2n+1}\right)=40\beta(4)+\frac{5}{12}\pi \zeta(3),\label{conj2}\\
&\sum_{n=0}^\infty\frac{\binom{2n}{n}}{(2n+1)^2(-16)^n}\left(5H_{2n+1}+\frac{12}{2n+1}\right)=14\zeta(3).\label{conj1}
\end{align}
\end{con}

There are two major steps in our proof of this conjecture. First, we will express
the Ap\'ery-type series on the left-hand side of \eqref{conj2} (resp.~\eqref{conj1}) by some log-sine-cosine (resp.~log-sinh-cosh) integrals. Then, we will evaluate these integrals using a single-valued version of the polylogarithms, denoted by $\tD_m(x)$ in Zagier's seminal paper \cite{Zag1990a}.

\section{Log-sine-cosine integrals}

\begin{defn}
Let $j$ and $k$ be two positive integers. For any real number $\gth$, we define the \emph{log-sine integrals} by
\begin{align*}
\Ls_{j}(\gth):=-\int_0^\gth\log^{j-1}\left|2\sin\frac{t}{2}\right| \dd t,
\end{align*}
and more generally, the \emph{log-sine-cosine integrals} by
\begin{align*}
\Lsc_{j,k}(\gth):=-\int_0^\gth\log^{j-1}\left|2\sin\frac{t}{2}\right|
    \log^{k-1}\left|2\cos\frac{t}{2}\right|\dd t.
\end{align*}

Similarly, for any real number $\gth$, we define the \emph{log-sinh integrals} by
\begin{align*}
\Lsh_{j}(\gth):=-\int_0^\gth\log^{j-1}\left|2\sinh\frac{t}{2}\right|\dd t,
\end{align*}
and more generally, the \emph{log-sinh-cosh integrals} by
\begin{align*}
\Lshch_{j,k}(\gth):=-\int_0^\gth\log^{j-1}\left|2\sinh\frac{t}{2}\right|
\log^{k-1}\left|2\cosh\frac{t}{2}\right|\dd t.
\end{align*}
\end{defn}

\bigskip

The log-sine-cosine integrals have been considered by L. Lewin \cite{Lew1958,Lew1981}.
They appear in physical applications as well, see for instance \cite{DK2004}.

\bigskip

The following simple fact is useful. For any positive integers $p$ and $n$, and for any nonnegative real number $z$, we have
\begin{align}\label{eqn-simple-fact}
\frac{1}{(p-1)!}\int_{0}^{z}\frac{\log^{p-1}\left(\frac{z}{w}\right)}{w}\cdot w^n\dd w=\frac{z^n}{n^p}.
\end{align}

\begin{lem} For any nonnegative integer $p$ and real number $z\in [0,1/2]$, we have
\begin{align}\label{eqn-lsc-1}
\sum_{n=0}^\infty\binom{2n}{n}\frac{z^{2n+1}}{(2n+1)^{p+1}}=\frac{\gth}{2}\frac{\log^p(2\sin\gth)}{p!}+\frac{1}{4p!}\sum_{j=1}^{p}(-1)^{j-1}\binom{p}{j}\log^{p-j}(2\sin\gth)\Ls_{j+1}(2\gth),
\end{align}
where $\gth:=\arcsin(2z) \in [0,\pi/2]$.

Similarly, for any nonnegative integer $p$ and real number $z\in [0,1/2]$, we have
\begin{align}\label{eqn-lshch-1}
\!\!\!\!\sum_{n=0}^\infty\binom{2n}{n}\frac{(-1)^{n}z^{2n+1}}{(2n+1)^{p+1}}=\frac{\gth}{2}\frac{\log^p(2\sinh\gth)}{p!}+\frac{1}{4p!}\sum_{j=1}^p(-1)^{j-1}\binom{p}{j}\log^{p-j}(2\sinh\gth)\Lsh_{j+1}(2\gth),
\end{align}
where $\gth:=\arcsinh(2z) \in [0,\log(\sqrt{2}+1)]$.
\end{lem}

\begin{proof}
The first identity \eqref{eqn-lsc-1} is proved in \cite[Theorem 4]{CPXZ2022}. For \eqref{eqn-lshch-1}, we start with the simple identity
\[\sum_{n=0}^\infty\binom{2n}{n}(-1)^{n}z^{2n+1}=\frac{z}{\sqrt{1+4z^2}}=\frac{\tanh\gth}{2}\qquad(\text{recall~}z=\frac{1}{2}\sinh\gth).\]
By \eqref{eqn-simple-fact}, we have
\begin{align}
\sum_{n=0}^\infty\binom{2n}{n}\frac{(-1)^{n}z^{2n+1}}{(2n+1)^{p+1}}&=\frac{1}{p!}\int_{0}^{z}\frac{\log^{p}\left(\frac{z}{w}\right)}{w}\cdot\sum_{n=0}^\infty\binom{2n}{n}(-1)^{n}w^{2n+1}\dd w\nonumber\\
&=\frac{1}{p!}\int_{0}^{\gth}\frac{\log^{p}\left(\frac{1}{2}\sinh \gth\big/\frac{1}{2}\sinh t\right)}{\frac{1}{2}\sinh t}\cdot\frac{\tanh t}{2}\dd \left(\frac{1}{2}\sinh t\right)\tag{$w=\frac{1}{2}\sinh t$}\\
&=\frac{1}{2p!}\int_{0}^{\gth}\left(\log(2\sinh \gth)-\log(2\sinh t)\right)^p\dd t\nonumber\\
&=\frac{\gth}{2}\frac{\log^p(2\sinh\gth)}{p!}+\frac{1}{2p!}\sum_{j=1}^{p}(-1)^{j}\binom{p}{j}\log^{p-j}(2\sinh\gth)\int_{0}^{\gth}\log^{j}(2\sinh t)\dd t\nonumber\\
&=\frac{\gth}{2}\frac{\log^p(2\sinh\gth)}{p!}+\frac{1}{4p!}\sum_{j=1}^p(-1)^{j-1}\binom{p}{j}\log^{p-j}(2\sinh\gth)\Lsh_{j+1}(2\gth).\nonumber
\end{align}
The proof is now complete.
\end{proof}

\bigskip

\begin{lem} For any positive integer $p$ and real number $z\in[0,1/2]$, we have
\begin{equation}\label{eqn-lsc-2}\begin{aligned}[c]
\sum_{n=1}^\infty \binom{2n}{n}\frac{H_{2n}}{(2n+1)^p}z^{2n+1}=\frac1{(p-1)!}\sum_{j=1}^{p}&(-1)^{j-1}\binom{p-1}{j-1}\log^{p-j}(2\sin\gth)\\[-2ex]
&\times\bigg\{\frac1{2}\Lsc_{j,2}(2\gth)-\sum_{l=1}^{j}\binom{j-1}{l-1}\Lsc_{l,j-l+2}(\gth)\bigg\},
\end{aligned}
\end{equation}
where $\gth=\arcsin(2z)\in[0,\pi/2]$.

Similarly, for any positive integer $p$ and real number $z\in [0,1/2]$, we have
\begin{equation}\label{eqn-lshch-2}
\begin{aligned}[c]
\sum_{n=1}^\infty \binom{2n}{n} \frac{H_{2n}}{(2n+1)^p}(-1)^{n}z^{2n+1}= {} &\frac1{(p-1)!} \sum_{j=1}^p (-1)^{j-1} \binom{p-1}{j-1} \log^{p-j}(2\sinh\gth)\\[-1ex]
&\quad \times\bigg\{\frac1{2}\Lshch_{j,2}(2\gth)-\sum_{l=1}^j \binom{j-1}{l-1}\Lshch_{l,j-l+2}(\gth)\bigg\},
\end{aligned}
\end{equation}
where $\gth=\arcsinh(2z)\in[0,\log(\sqrt{2}+1)]$.
\end{lem}
\begin{proof}
We first prove \eqref{eqn-lsc-2}. By \cite[Eqn. (D.8), pp. 52--53]{DK2004}, we have
\begin{align}
\sum_{n=1}^{\infty}\binom{2n}{n}\frac{z^{n}}{2n}&=\log (1+\chi),\\
\sum_{n=1}^{\infty}\binom{2n}{n}H_{2n-1}z^{n}&=\frac{2}{1-\chi}[\chi\log(1+\chi)-(1+\chi)\log(1-\chi)],
\end{align}
where $\chi:=\frac{1-\sqrt{1-4z}}{1+\sqrt{1-4z}}$. Summing up the two equations, substituting $z^2$ for $z$ and then multiplying by $z$, we obtain
\begin{align}\label{eqn-1}
\sum_{n=1}^{\infty}\binom{2n}{n}H_{2n}z^{2n+1}=\frac{z}{\sqrt{1-4z^2}}\left(\log\left(\frac{2}{1+\sqrt{1-4z^2}}\right)-2\log\left(\frac{2\sqrt{1-4z^2}}{1+\sqrt{1-4z^2}}\right)\right).
\end{align}
By the change of variables $z=\frac{1}{2}\sin \gth$ for $z \in [0,1/2)$ and $\gth \in [0,\pi/2)$, we arrive at
\[\sum_{n=1}^{\infty}\binom{2n}{n}H_{2n}z^{2n+1}=\tan\gth\cdot\left(\log\left(2\cos\frac{\gth}{2}\right)-\log\left(2\cos\gth\right)\right).\]
Using \eqref{eqn-simple-fact}, we obtain
\begin{align}
&\sum_{n=1}^{\infty}\binom{2n}{n}\frac{H_{2n}}{(2n+1)^p}z^{2n+1}=\frac{1}{(p-1)!}\int_{0}^{z}\frac{\log^{p-1}\left(\frac{z}{w}\right)}{w}\cdot\sum_{n=1}^{\infty}\binom{2n}{n}H_{2n}w^{2n+1}\dd w\nonumber\\
&=\frac{1}{(p-1)!}\int_{0}^{\gth}\left(\log\left(2\sin\gth\right)-\log\left(2\sin t\right)\right)^{p-1}\cdot\left(\log\!\left(2\cos\frac{t}{2}\right)-\log\left(2\cos t\right)\right)\dd t
\,\,\tag{$w=\frac{1}{2}\sin t$}\\
&\begin{aligned}[c] = \frac{1}{(p-1)!}\sum_{j=1}^{p}\begin{aligned}[t]
 (-1)^{j-1} & \binom{p-1}{j-1} \log^{p-j}\left(2\sin\gth\right) \\
& \times \int_{0}^{\gth}\log^{j-1}\left(2\sin t\right)\cdot\left(\log\left(2\cos\frac{t}{2}\right)-\log\left(2\cos t\right)\right)\dd t. \label{eq-Harm-Binc-cent-1}
\end{aligned}\end{aligned}
\end{align}
We observe that
\begin{align}
\int_{0}^{\gth}\log^{j-1}\left(2\sin t\right)\cdot\left(-\log\left(2\cos t\right)\right)\dd t=\frac1{2}\Lsc_{j,2}(2\gth),\label{eq-Harm-Binc-cent-2}
\end{align}
and
\begin{align}
&\int_{0}^{\gth}\log^{j-1}\left(2\sin t\right)\cdot\log\left(2\cos\frac{t}{2}\right)\dd t\nonumber\\
&=\int_{0}^{\gth}\left(\log\left(2\sin\frac{t}{2}\right)+\log\left(2\cos\frac{t}{2}\right)\right)^{j-1}\cdot\log\left(2\cos\frac{t}{2}\right)\dd t\nonumber\\
&=\int_{0}^{\gth}\sum_{l=1}^{j}\binom{j-1}{l-1}\log^{l-1}\left(2\sin\frac{t}{2}\right)\log^{j-l+1}\left(2\cos\frac{t}{2}\right)\dd t\nonumber\\
&=-\sum_{l=1}^j \binom{j-1}{l-1}\Lsc_{l,j-l+2}(\gth).\label{eq-Harm-Binc-cent-3}
\end{align}
Inserting \eqref{eq-Harm-Binc-cent-2} and \eqref{eq-Harm-Binc-cent-3} in \eqref{eq-Harm-Binc-cent-1}, we complete the proof of \eqref{eqn-lsc-2} for $z\in[0,1/2)$. The case $z=1/2$ follows from continuity.

The proof of \eqref{eqn-lshch-2} is similar. In fact, note that \eqref{eqn-1} is valid for all complex numbers $z$ with $|z|<1/2$. If we substitute $\mi z$ for $z$ in \eqref{eqn-1}, we have
\[\sum_{n=1}^{\infty}\binom{2n}{n}H_{2n}(-1)^{n}z^{2n+1}=\frac{z}{\sqrt{1+4z^2}}\left(\log\left(\frac{2}{1+\sqrt{1+4z^2}}\right)-2\log\left(\frac{2\sqrt{1+4z^2}}{1+\sqrt{1+4z^2}}\right)\right).\]
Let $z=\frac{1}{2}\sinh\gth$ for $z\in[0,1/2)$ and $\gth\in[0,\log(\sqrt{2}+1))$, the above equation can be written as
\[\sum_{n=1}^{\infty}\binom{2n}{n}H_{2n}(-1)^{n}z^{2n+1}=\tanh\gth\cdot\left(\log\left(2\cosh\frac{\gth}{2}\right)-\log\left(2\cosh\gth\right)\right).\]
Therefore,
\begin{multline*}
\sum_{n=1}^{\infty}\binom{2n}{n}\frac{H_{2n}}{(2n+1)^p}(-1)^{n}z^{2n+1}=\frac{1}{(p-1)!}\int_{0}^{z}\frac{\log^{p-1}\left(\frac{z}{w}\right)}{w}\cdot\sum_{n=1}^{\infty}\binom{2n}{n}H_{2n}(-1)^{n}w^{2n+1}\dd w \\
=\frac{1}{(p-1)!}\int_{0}^{\gth} \big(\log\left(2\sinh\gth\right)-\log\left(2\sinh t\right)\big)^{p-1} \left(\log\left(2\cosh\frac{t}{2}\right)-\log\left(2\cosh t\right)\right)\dd t
 \end{multline*}
by substitution $w=\frac{1}{2}\sinh t$.
Identity \eqref{eqn-lshch-2} follows from expanding $\left(\log\left(2\sinh\gth\right)-\log\left(2\sinh t\right)\right)^{p-1}$ by the binomial theorem.
\end{proof}

\section{Proof of \protect\eqref{conj2}}
Setting $p=3$ and $z=1/4$ in \eqref{eqn-lsc-1} and \eqref{eqn-lsc-2}, we have $\gth=\pi/6$ and
\begin{align}
\sum_{n=0}^\infty\frac{\binom{2n}{n}}{(2n+1)^{4}16^n}&=\frac{1}{6}\Ls_4\left(\frac{\pi}{3}\right),\label{eqn-4-1}\\
\sum_{n=1}^\infty\frac{\binom{2n}{n}H_{2n}}{(2n+1)^{3}16^n}&=\Lsc_{3,2}\left(\frac{\pi}{3}\right)-2\Lsc_{3,2}\left(\frac{\pi}{6}\right)-4\Lsc_{2,3}\Big(\frac{\pi}{6}\Big)-2\Lsc_{1,4}\Big(\frac{\pi}{6}\Big).\label{eqn-4-2}
\end{align}
For ease of reading, we now outline our proof as the calculations are somewhat involved. We first express the functions $\Lsc_{j,k}(\theta)$ ($j+k=5$) for $j<k$ in terms of the ones with $j>k$ and then show that we can rewrite each of the latter, after subtracting a suitable linear term in $\theta$, in terms of the single-valued function $\tD_{4}$ (Lemmas \ref{lem-4-1} and \ref{lem-4-2}). Substituting $\theta = \frac{\pi}{6}$ and $= \frac{5\pi}{6}$ as in \eqref{eqn-Ls4-5Pi6}--\eqref{eqn-Lsc32-5Pi6}  then reveals that the ensuing rational multiples of $\pi\zeta(3)$  indeed conspire to match the one on the RHS of \eqref{conj2}. Moreover, upon realizing that $\beta(4)$ can be written as $\tD_4({\mathrm i})$, Conjecture \eqref{conj2} is reduced to showing the vanishing of a rational linear combination of only $\tD_4$-terms as in \eqref{TBP-0}.
It then remains to find---and in fact to concoct---suitable functional equations for $\tD_4$ which, after an appropriate specialization, match precisely this combination.

\medskip
\emph{Step 1}. It is clear from the definition that
\[ \Lsc_{j,k}(\gth)=\Lsc_{j,k}(\pi)-\Lsc_{k,j}(\pi-\gth), \quad \gth\in[0,\pi].  \]
The special values of $\Lsc_{j,k}$ at $\pi$ have been determined by L. Lewin \cite{Lew1958} and \cite[Section 7.9]{Lew1981}. As observed in \cite{BBSW2012}, Lewin's result can be stated in the form
\[-\frac{1}{\pi}\sum_{m, n=0}^{\infty}\Lsc_{m+1,n+1}(\pi)\frac{x^{m}}{m!}\frac{y^{n}}{n!}=\frac{2^{x+y}}{\pi} \frac{\Gamma\big(\frac{1+x}{2}\big)\Gamma\big(\frac{1+y}{2}\big)}{\Gamma\big(1+\frac{x+y}{2}\big)}.\]
In particular, it is known that
\begin{alignat}{2}
\Lsc_{1,4}(\gth) &=-\Ls_4(\pi-\gth)+\Ls_4(\pi),\quad & \Ls_4(\pi)&=\frac{3}{2}\pi\zeta(3),\label{eqn-Lsc14}\\
\Lsc_{2,3}(\gth) &=-\Lsc_{3,2}(\pi-\gth)+\Lsc_{3,2}(\pi),\quad & \Lsc_{3,2}(\pi)&=-\frac{1}{4}\pi\zeta(3)\label{eqn-Lsc23}.
\end{alignat}
Inserting \eqref{eqn-Lsc14} and \eqref{eqn-Lsc23} in \eqref{eqn-4-2}, we have
\begin{align}
\sum_{n=1}^\infty\frac{\binom{2n}{n}H_{2n}}{(2n+1)^{3}16^n}
=2\Ls_{4}\Big(\frac{5\pi}{6}\Big)+\Lsc_{3,2}\Big(\frac{\pi}{3}\Big)
-2\Lsc_{3,2}\Big(\frac{\pi}{6}\Big)+4\Lsc_{3,2}\Big(\frac{5\pi}{6}\Big)-2\pi\zeta(3).\label{eqn-4-3}
\end{align}

\emph{Step 2}. We introduce two different versions of the Bloch-Wigner-Ramakrishnan-Wojtkowiak-Zagier polylogarithm \cite{Wojtkowiak1989,Zag1990,Zag1990a}: for $|x|\le 1$, $x\ne 0,1$,
\begin{align}
D_{m}(x)={}& \RR_{m}\bigg(\sum_{j=0}^{m}\frac{(-\log |x|)^{m-j}}{(m-j)!}\Li_{j}(x)\bigg), \label{defn:D}  \\
\tD_{m}(x)={}& D_{m}(x)+(1-(-1)^m) \frac{\log^{m-1} |x|}{4\cdot m!}(2\log|1-x|-\log|x|)    \label{defn:tD} \\
={}&\RR_{m}\bigg(\sum_{j=1}^{m}\frac{(-\log |x|)^{m-j}}{(m-j)!}\Li_{j}(x)
+ \frac{\log^{m-1}|x|}{m!}\log|1-x| \bigg) ,   \notag
\end{align}
where  $\RR_{m}=\Im$ for $m$ even and $\RR_{m}=\Re$ for $m$ odd, and where we adopt Zagier's ad hoc convention (p.413 in loc.cit.) $\Li_{0}(x)\equiv -1/2$. It is easy to see that
\begin{equation}\label{equ:tD=0at0}
\lim_{x\to 0} \tD_{m}(x)=0.
\end{equation}
We extend $\tD_{m}(x)$ to $\mathbb{C}\setminus\{0,1\}$ as a single-valued and real analytic function by
the \emph{inversion relation} \eqref{D4-1} and we can check that $\tD_{m}(x)$ satisfies the \emph{complex conjugate relation} \eqref{D4-2} below
\begin{align}
\tD_m(x)&=(-1)^{m-1} \tD_m(x^{-1}),     \label{D4-1}\\
\tD_m(x)&=(-1)^{m-1} \tD_m(\overline{x}).  \label{D4-2}
\end{align}
In particular, complex conjugate relation implies that
\begin{align}
\tD_{2m}(x)=0 \quad \text{for all }x\in\R.  \label{tDm-vanishOnR}
\end{align}
It also satisfies \emph{distribution relations} as follows: for any positive integer $N$ we have
\begin{align}
\tD_m(x^N)&=N^{m-1}\sum_{j=0}^{N-1} \tD_m\big(x e^{2j\pi\mi/N}\big) .\label{tDm-distribution}
\end{align}
Indeed, this follows easily from the fact that for all $|x|\le 1$ and $1\le j\le m$ we have
\begin{align*}
\log^{m-j}|x^N|\Li_{j}(x^N)&\, =N^{m-1}\log^{m-j}|x| \sum_{j=0}^{N-1} \Li_{j}\big(x e^{2j\pi\mi/N}\big)  ,\\
1-x^N &\, =\prod_{j=0}^{N-1} \big(1-x e^{2j\pi\mi/N}\big).
\end{align*}

\medskip
The following computational lemma will be used repeatedly below.  
\begin{lem}\label{lem-dDm}
Let $0<\gth<\pi$. Let $f(x)$ be a rational function of $x$ with real coefficients. Set
\begin{equation*}
g_f(\gth)=\frac12 \frac{\dd}{\dd\gth} \log f(e^{\mi\gth}) =\frac{\mi e^{\mi\gth} f'(e^{\mi\gth})}{2f(e^{\mi\gth})},\quad
h_f(\gth)= \frac{\dd}{\dd\gth} \Li_1 \big(f(e^{\mi\gth})\big)
=\frac{\mi e^{\mi\gth} f'(e^{\mi\gth})}{1-f(e^{\mi\gth}) }.
\end{equation*}
For any positive integer $m$ let $\gs_m=2\mi, \gd_m=0$ if $m$ is even and $\gs_m=2,\gd_m=1$ if $m$ is odd.
Then
\begin{align*}
\frac{\dd}{\dd\gth}D_{m} \big(f(e^{\mi\gth})\big)
\,=\,{} & (-1)^m \left(\!D_{m-1}\big(f(e^{\mi\gth})\big) - \RR_{m-1}\frac{\log^{m-1}|f(e^{\mi\gth})|}{2(m-1)!} \right)
\frac{g_f(\gth)+g_f(-\gth)}{\mi} \\
&\,+ \frac{\big({-}\log| f(e^{\mi\gth})|\big)^{m-1}}{\gs_m\cdot (m-1)!}
 \left(  \gd_m \Big(g_f(\gth)-g_f(-\gth)\Big) + h_f(\gth) +(-1)^m h_f(-\gth) \right), \end{align*}
 \begin{align*}
\frac{\dd}{\dd\gth}\tD_{m} \big( & f(e^{\mi\gth})\big)
= {} \\
& (-1)^m \left(\tD_{m-1}\big(f(e^{\mi\gth})\big) - \RR_{m-1}\frac{\log^{m-2}|f(e^{\mi\gth})|\log|1-f(e^{\mi\gth})|}{(m-1)!} \right)\frac{g_f(\gth)+g_f(-\gth)}{\mi} \\
&\,+ \frac{\big({-}\log| f(e^{\mi\gth})|\big)^{m-1}}{\gs_m\cdot (m-1)!}
 \left( h_f(\gth) +(-1)^m h_f(-\gth) \right)
+ \gd_m \frac{\log^{m-1}\big|f(e^{\mi\gth})\big|}{2\cdot m!} \big(h_f(-\gth)-h_f(\gth)\big) \\
&\,+  \gd_m \frac{(m-1)\log^{m-2} |f(e^{\mi\gth})|\log|1-f(e^{\mi\gth})|}{m!} \Big(g_f(\gth)-g_f(-\gth)\Big).
\end{align*}

\end{lem}

\begin{proof} By definition, we may rewrite $D_{m}\big(f(e^{\mi\gth})\big)$ as
\begin{align*}
D_{m}\big(f(e^{\mi\gth})\big)=&\, \sum_{j=0}^{m}\frac{\big(\tfrac{1}{2}\big({-}\log f(e^{\mi\gth})-\log f(e^{-\mi\gth})\big)\big)^{m-j}}{(m-j)!}
\cdot\frac{\Li_{j}\big(f(e^{\mi\gth})\big)-(-1)^m\Li_{j}\big(f(e^{-\mi\gth})\big)}{\gs_m}.
\end{align*}
Thus we have
\begin{align*}
& \frac{\dd}{\dd\gth}D_{m}\big(f(e^{\mi\gth})\big) \,=\, \\
& \frac{1}{\gs_m}\sum_{j=0}^{m-1}
    \frac{\big({-}\log| f(e^{\mi\gth})|\big)^{m-1-j}}{(m-1-j)!} \Big({-}g_f(\gth)+g_f(-\gth) \Big)
    \left(\Li_{j}\big(f(e^{\mi\gth})\big)-(-1)^m \Li_{j}\big(f(e^{-\mi\gth})\big) \right)   \\
& {} + \frac{1}{\gs_m}\sum_{j=2}^{m}\frac{\big({-}\log| f(e^{\mi\gth})|\big)^{m-j}}{(m-j)!}  \left( 2g_f(\gth)\Li_{j-1}\big(f(e^{\mi\gth})\big)
+(-1)^m  2g_f(-\gth)\Li_{j-1}\big(f(e^{-\mi\gth})\big)  \right)\\
&{} +\frac{1}{\gs_m}\frac{\big({-}\log |f(e^{\mi\gth})|\big)^{m-1}}{(m-1)!} \Big(h_f(\gth)+(-1)^m h_f(-\gth)\Big).
\end{align*}
Moving the $j=0$ term in the first sum to the end,
setting $j\to j+1$ in the second sum, and combining like terms, we then arrive at
\begin{align*}
&\, \frac{\dd}{\dd\gth}D_{m}\big(f(e^{\mi\gth})\big) \\[2ex]
& {} = \begin{aligned}[t]
& \frac{1}{\gs_m}\sum_{j=1}^{m-1}
    \frac{\big({-}\log| f(e^{\mi\gth})|\big)^{m-1-j}}{(m-1-j)!} \Big(g_f(\gth)+g_f(-\gth) \Big)
    \left(\Li_{j}\big(f(e^{\mi\gth})\big)+(-1)^m \Li_{j}\big(f(e^{-\mi\gth})\big) \right)   \\[-0.5ex]
&{} + \frac{1}{\gs_m} \frac{\big(-\log |f(e^{\mi\gth})|\big)^{m-1}}{(m-1)!} \Big( \gd_m  \Big(g_f(\gth)-g_f(-\gth)\Big) + h_f(\gth)+(-1)^m h_f(-\gth)\Big) \end{aligned} \\[2ex]
& {} = \begin{aligned}[t]
& (-1)^m \sum_{j=1}^{m-1}
    \frac{\big({-}\log| f(e^{\mi\gth})|\big)^{m-1-j}}{(m-1-j)!}
 \RR_{m-1}\Big(\Li_{j}\big(f(e^{\mi\gth})\big) \Big)\frac{g_f(\gth)+g_f(-\gth)}{\mi}    \\[-0.5ex]
&{} +\frac{1}{\gs_m} \frac{\big({-}\log |f(e^{\mi\gth})|\big)^{m-1}}{(m-1)!} \Big( \gd_m \Big(g_f(\gth)-g_f(-\gth)\Big) + h_f(\gth)+(-1)^m h_f(-\gth)\Big). \end{aligned}
\end{align*}
The expression for $D_m$ in the lemma now follows easily from the definition \eqref{defn:D}.

Turning to $\tD_m$, we only need to handle the extra term at the end of \eqref{defn:tD}.
Noticing that $2\log|1-x|=-\Li_1(x)-\Li_1(\bar{x})$ we have
\begin{align*}
   & \hspace{-2em} \frac{\dd}{\dd\gth} \frac{\log^{m-1}\big|f(e^{\mi\gth})\big|}{2\cdot m!} \Big(2\log\big|1-f(e^{\mi\gth})\big|-\log\big|f(e^{\mi\gth})\big|\Big)   \\[2ex]
\, = \,   & \frac{\log^{m-2}\big|f(e^{\mi\gth})\big|}{2\cdot m!}
\Big(2(m-1) \log\big|1-f(e^{\mi\gth})\big|
-m\log\big|f(e^{\mi\gth})\big|\Big)\big(g_f(\gth)-g_f(-\gth)\big) \\
& {} + \frac{\log^{m-1}\big|f(e^{\mi\gth})\big|}{2\cdot m!} \big(h_f(-\gth)-h_f(\gth)\big) .
\end{align*}
Now we can complete the proof of the lemma immediately.
\end{proof}

\begin{cor} \label{cor-dDm}
Notation as above. Put
$$A(\gth)=\log \Big|2\sin \frac{\gth}{2}\Big|= \log\big|1-e^{\mi\gth}\big|,\quad
B(\gth)=\log\Big|2\cos \frac{\gth}{2}\Big|=\log\big|1+e^{\mi\gth}\big|.$$
For any positive integer $m$ let $a^{\pm}_m(\gth)=1$ if $m$ is even,
and $a^{\pm}_m(\gth)=\mi (1\mp e^{\mi\gth})/(1\pm e^{\mi\gth})$ if $m$ is odd.
Then for all $m\ge 3$
\begin{align*}
\frac{\dd}{\dd\gth} A(\gth)& {} = -\frac{a_1^-}2, \qquad \frac{\dd}{\dd\gth} B(\gth)=-\frac{a_1^+}2,\\[1ex]
\frac{\dd}{\dd\gth} \tD_m\big({\pm} e^{\mi\gth}\big) & {} = (-1)^{m}  \tD_{m-1} \big({\pm}  e^{\mi\gth}\big) , \\[1ex]
\frac{\dd}{\dd\gth} \tD_m\big(1\pm e^{\mi\gth} \big)
& {} = \frac{(-1)^m}2 \tD_{m-1}\big(1\pm e^{\mi\gth} \big)
+(1+(-1)^m)\frac{A_\pm^{m-1}}{2\cdot (m-1)!} \qquad(A_+=B, A_-=A), \\[1ex]
\frac{\dd}{\dd\gth} \tD_m\left(\frac{1-e^{\mi\gth}}{1+e^{\mi\gth}}\right)
& {} =  \begin{aligned}[t]
& \frac{\gd_m(A-B)^{m-2}}{2\cdot m!} \Big((A-B) a_1^+  + (m-1)(\log2-B)(a_1^+- a_1^-)\Big) \\
& + \frac{(B-A)^{m-1}}{2\cdot (m-1)!} a_m^+
. \end{aligned}
\end{align*}
\end{cor}
\begin{proof} By simple calculations,
\begin{align*}
f(x)=1\pm x: &\,  \qquad \left\{\aligned
g_f(\gth)+g_f(-\gth) ={} & \frac{\pm \mi e^{\mi\gth}}{2(1\pm e^{\mi\gth})}+\frac{\pm \mi e^{-\mi\gth}}{2(1\pm e^{-\mi\gth})}=\frac{\mi}2,\\
g_f(\gth)-g_f(-\gth) ={} & \frac{\dd}{\dd\gth} (A\ \text{or}\ B)=-\frac{a_1^\pm }2,\\
h_f(\gth)+h_f(-\gth) ={} & {-}2\mi, \qquad  h_f(\gth)-h_f(-\gth)=0,  \phantom{\frac12}
\endaligned \right.
\\
f(x)=\frac{1-x}{1+x}: &\,  \qquad  \left\{\aligned
g_f(\gth)+g_f(-\gth)={} & \frac{-\mi e^{\mi\gth}}{1-e^{2\mi\gth}}+\frac{-\mi e^{-\mi\gth}}{1-e^{-2\mi\gth}}=0,\\
g_f(\gth)-g_f(-\gth)={} & \frac{\dd}{\dd\gth} (A(\gth)-B(\gth))=\frac{a_1^+}2-\frac{a_1^-}2,\\
h_f(\gth)+h_f(-\gth)={} & \mi , \qquad  h_f(\gth)-h_f(-\gth)=a_1^+.  \phantom{\frac12}
\endaligned  \right.
\end{align*}
Hence
\begin{align*}
\frac{\dd}{\dd\gth} \tD_m\big(1\pm e^{\mi\gth} \big)
& {} =\frac{(-1)^m}2 \tD_{m-1}\big(1\pm e^{\mi\gth} \big)
+(1+(-1)^m)\frac{\log^{m-1}|1\pm e^{\mi\gth}|}{2\cdot (m-1)!} , \\[1ex]
\frac{\dd}{\dd\gth} \tD_m\left(\frac{1-e^{\mi\gth}}{1+e^{\mi\gth}}\right)
& {} = \begin{aligned}[t]
& \Big(\!{-}\!\log\Big|\frac{1-e^{\mi\gth}}{1+e^{\mi\gth}}\Big|\Big)^{m-1}
\cdot \frac{h_f(\gth) +(-1)^m h_f(-\gth) }{(m-1)!\cdot \gs_m}
+ \frac{\gd_m a_1^+}{2\cdot m!} \log^{m-1}\Big|\frac{1-e^{\mi\gth}}{1+e^{\mi\gth}}\Big|\\[-0.5ex]
& {} + \frac{\gd_m}{2m!}   (m-1)\log^{m-2} \Big|\frac{1-e^{\mi\gth}}{1+e^{\mi\gth}}\Big|\log\Big|\frac{2}{1+e^{\mi\gth}}\Big| \big(a_1^+- a_1^-\big).
\end{aligned}
\end{align*}
These quickly lead to the equalities in the corollary.
\end{proof}

\emph{Step 3}. Next, we express both $\Ls_4$ and $\Lsc_{3,2}$ in terms of polylogarithms.
\begin{lem}\label{lem-4-1}
The following expression for $\Ls_4(\gth)$ holds for all $\gth \in (0,\pi)$:
\begin{align}
\Ls_{4}(\gth)= \frac32 \zeta(3)\gth+\frac32\Big\{{-}\tD_{4}\big(e^{\mi\gth}\big)
  -4\tD_{4}\big(1-e^{\mi\gth}\big)\Big\}.       \label{equ:Ls4}
\end{align}
\end{lem}

\begin{proof} First we observe that $\tD_{4}(1)=0$ by \eqref{tDm-vanishOnR}.
Thus, taking $\gth\to 0$ we see that it suffices to prove the equality of the derivatives of
both sides of \eqref{equ:Ls4}. Since
\begin{equation*}
 \frac{\dd}{\dd\gth} \Ls_{4}(\gth)=-\log^3\Big|2\sin \frac{\gth}{2}\Big|=-A^3,
\end{equation*}
by Corollary~\ref{cor-dDm} we have
\begin{align}
& \frac{\dd}{\dd\gth}\left\{\frac23\Ls_{4}(\gth)-\zeta(3)\gth+\tD_{4}\big(e^{\mi\gth}\big)
    +4\tD_{4}\big(1-e^{\mi\gth}\big)\right\}  \nonumber\\
& {} = -\zeta(3) + \frac{\Li_3 (e^{\mi\gth} )+\Li_3 (e^{-\mi\gth} ) }{2}
+2\tD_3\big(1-e^{\mi\gth}\big). \label{equ:Ls3Step2}
\end{align}
Since $\Li_3(1)=\zeta(3)$ and $\lim_{\gth\to 0}  \tD_3\big(1-e^{\mi\gth}\big) =0$ by \eqref{equ:tD=0at0},
it suffices to prove the derivative of \eqref{equ:Ls3Step2} vanishes.
Clearly $\tD_m(x)=D_m(x)$ for all even $m$ by \eqref{defn:tD}.
Thus, using Corollary \ref{cor-dDm} again we see that
\begin{align}
 \frac{\dd}{\dd\gth}{\big(\mathrm {RHS\ of\ }}\eqref{equ:Ls3Step2}\big)
= {} &
-\frac{\Li_2 (e^{\mi\gth} )-\Li_2 (e^{-\mi\gth} ) }{2i} -\tD_2 \big(1-e^{\mi\gth}\big)    \notag\\
= {} & -\tD_2 (e^{\mi\gth})  - \tD_2 \big(1-e^{\mi\gth}\big)
\,=\,-D_2 (e^{\mi\gth}) - D_2 \big(1-e^{\mi\gth}\big) \,=\, 0 \label{equ:completeD2}
\end{align}
by  \cite[Eqn.~(4)]{Zag1990a}.  This completes the proof of Lemma \ref{lem-4-1}.
\end{proof}

\begin{lem}\label{lem-4-2}
The following expression for $\Lsc_{3,2}(\gth)$ holds for all $\gth \in (0,\pi)$:
\begin{align}
 \begin{aligned}[c]
 \Lsc_{3,2}(\gth) = {}
-\frac{\zeta(3)}{4}\gth & -\frac{1}{2}\tD_{4}\big({-}e^{\mi\gth}\big)
-\tD_{4}\big(e^{\mi\gth}\big)
+2\tD_{4}\big(1+e^{\mi\gth}\big) \\[-0.5ex]
& +2\tD_{4}\Big(\frac{1-e^{\mi\gth}}{1+e^{\mi\gth}}\Big)
-\frac{1}{2}\tD_{4}\big(1-e^{2\mi\gth}\big) .  \end{aligned} \label{equ:Lsc32}
\end{align}
\end{lem}

\begin{proof} The proof of this lemma is completely similar to that of Lemma \ref{lem-4-1}.
As above, let $A=A(\gth)$ and $B=B(\gth)$.
By straightforward computations using Corollary~\ref{cor-dDm} we find that
\begin{align*}
\frac{\dd}{\dd\gth} \Lsc_{3,2}(\gth) = {} & - \log^2\Big|2\sin \frac{\gth}{2}\Big|\log\Big|2\cos \frac{\gth}{2}\Big|=-A^2 B,\\[1ex]
\frac{\dd}{\dd\gth} \tD_{4}\big(e^{\mi\gth}\big)
= {} &\tD_3 \big(e^{\mi\gth} \big)
\quad \xrightarrow{\dd/\dd\gth} \quad -\tD_2 \big(e^{\mi\gth}\big)    ,\\[1ex]
\frac{\dd}{\dd\gth} \tD_{4}\big({-}e^{\mi\gth}\big)
= {} &\tD_3 (-e^{\mi\gth} )
\quad \xrightarrow{\dd/\dd\gth} \quad  -\tD_2 (-e^{\mi\gth})   ,\\[1ex]
\frac{\dd}{\dd\gth} \tD_{4}\big(1+e^{\mi\gth}\big)
= {} &\frac{1}2 \tD_3\big(1+e^{\mi\gth} \big)+\frac{B^3}{6}
 \quad \xrightarrow{\dd/\dd\gth} \quad  {-}\frac{1}4 \tD_2\big(1+e^{\mi\gth} \big)-\frac{1}{4}B^2 a_1^+ ,\\[1ex]
    \frac{\dd}{\dd\gth} \tD_{4}\Big(\frac{1}{1+e^{\mi\gth}}\Big)
= {} & -\frac{ B^3}{12}
    -\frac{1}{2} \sum_{j=1}^{3} \frac{B^{3-j}}{(3-j)!}
    \left(\frac{\Li_{j}\big(\frac{1}{1+e^{\mi\gth}}\big)}{1+e^{\mi\gth}}
    +\frac{\Li_{j}\big(\frac{1}{1+e^{-\mi\gth}}\big)}{1+e^{-\mi\gth}} \right)\\[-1ex]
&  (\text{which is used to compute the limit as }\gth\to 0) ,\\[1ex]
\frac{\dd}{\dd\gth} \tD_{4}\Big(\frac{1-e^{\mi\gth}}{1+e^{\mi\gth}}\Big)
= {} &\frac{(B-A)^3}{12} \quad \xrightarrow{\dd/\dd\gth} \quad  \frac{ (A-B)^2(a_1^+ -a_1^-)}{8} ,\\[1ex]
\frac{\dd}{\dd\gth} \tD_{4}\big(1-e^{2\mi\gth}\big)
= {} &  \tD_3\big(1-e^{2\mi\gth} \big)  + \frac{ (A+B)^3}{3}
\quad \xrightarrow{\dd/\dd\gth} \quad  -\tD_2\big(1-e^{2\mi\gth} \big)-\frac{1}{2}(A+B)^2(a_1^+ +a_1^-)  .
\end{align*}
Thus by taking $\gth\to 0$ we see that the difference between the left-hand and right-hand sides of
$\frac{\dd}{\dd\theta}\big({\mathrm {RHS\ of\ }}\eqref{equ:Lsc32}\big)$ is
\begin{align*}
\frac{\log^3 (2)}{6}+\frac{5\zeta(3)}{4}+\frac12\Li_3(-1)
-\sum_{j=1}^{3} \frac{\log^{3-j}(2)}{(3-j)!}\Li_{j}\Big(\frac{1}{2}\Big) =0
\end{align*}
by the identities (see \cite[(1.16), (6.5) and (6.12)]{Lew1981})
\begin{align*}
\Li_2\Big(\frac12\Big)= {} & \frac12 \big(\zeta(2)-\log^2 (2)\big),\\
\Li_3\big({-}1\big) = {} & -\frac{3}{4} \zeta(3) , \\
\Li_3\Big(\frac{1}{2}\Big)= {} &  \frac{7}{8}\zeta(3)-\frac{1}{12}\pi^2\log(2)+\frac{1}{6}\log^3(2).
\end{align*}
Thus we only need to show the second derivatives of both sides of \eqref{equ:Lsc32} agree:
\begin{align}\label{equ:Lsc32Step1}
\begin{aligned}[c]
\frac{\dd}{\dd\gth} \bigg({-}2A^2B+\frac{\dd}{\dd\gth} \bigg\{
\begin{aligned}[t]
2\tD_{4}\big(e^{\mi\gth}\big)
& +\tD_{4}\big(-e^{\mi\gth}\big)
-4\tD_{4}\big(1+ e^{\mi\gth} \big)  \\
& {} +\tD_{4}\big(1-e^{2\mi\gth}\big)
-4\tD_{4}\bigg(\frac{1-e^{\mi\gth}}{1+e^{\mi\gth}}\bigg)
\! \bigg\}  \bigg) \stackrel{?}{=}0. \end{aligned} \end{aligned}
\end{align}
Now we have
\begin{align*}
\text{LHS of }\eqref{equ:Lsc32Step1}= {} &
\frac{\dd}{\dd\gth} \bigg( \begin{aligned}[t]
 -2A^2B
+2\tD_{3}\big(e^{\mi\gth}\big)
& {} +\tD_{3}\big({-}e^{\mi\gth}\big)
-4\bigg( \frac{1}2 \tD_3\big(1+e^{\mi\gth} \big)+\frac{B^3}{6} \bigg) \\[-0.5ex]
&
{} +  \tD_3\big(1-e^{2\mi\gth} \big) +\frac{ (A+B)^3}{3} -\frac{ (A-B)^3}{3} \bigg) \end{aligned} \\[1ex]
= {} & \begin{aligned}[t]
& A^2 a_1^+ + 2AB  a_1^- - \tD_2\big({-}e^{\mi\gth}\big)-2\tD_2\big(e^{\mi\gth}\big)
+\tD_2\big(1+e^{\mi\gth} \big)+ B^2 a_1^+  \\[-0.5ex]
&{}-\tD_2\big(1-e^{2\mi\gth} \big)-\frac{1}{2}(A+B)^2(a_1^+ +a_1^-)
 -\frac{1}{2}(A-B)^2(a_1^+ - a_1^-) \end{aligned} \\[1ex]
= {} & \tD_2\big(e^{2\mi\gth} \big)-2 \tD_2\big({-}e^{\mi\gth}\big)-2\tD_2\big(e^{\mi\gth}\big) \,=\, 0
\end{align*}
by \eqref{equ:completeD2} and then the distribution relation.
This completes the proof of the lemma.
\end{proof}

\emph{Step 4}. 
We will need the following functional equation of $\tD_4$, which is a variant of Kummer's $\Li_4$ equation \cite[Eqn. (7.78)]{Lew1981}. (Note $\Lambda_4(x)$ therein is closely related to $\mathrm{Li}_4(-x)$, in particular it only differs by products of lower weight terms.  In order to convert from \cite[Eqn. (7.78)]{Lew1981} to the \( \tD_4 \) functional equation we essentially only need to add a negative sign to all the arguments from \cite[Eqn. (7.78)]{Lew1981}, and drop any product terms.)

Let $\Q[\CP^1]$ be the set of finite $\Q$-linear combinations $\sum c_j[x_j]$ with $c_j\in\Q$, $x_j\in \CP^1$. We can then linearly
extend $\tD_m$ over $\Q[\CP^1]$.

\begin{lem}\label{lem-D4}(Kummer)
For any $x,y\in {\mathbb C}\setminus \{0,1\}$, set $\xi=\xi_x:=1-x$, $\eta=\eta_y:=1-y$ and
	\begin{align*}
	H(x,y) & {} :=  \begin{aligned}[t]
	&
	\Big[\frac{x^{2} y}{\eta^{2} \xi}\Big]
	+\Big[{-}\frac{\eta x^{2}y}{\xi}\Big]
	-3\Big[{-}\frac{x}{\eta\xi}\Big]
	-3\Big[{-}\frac{\eta x}{\xi}\Big]
	-3\Big[\frac{x}{\eta}\Big]
	-3\Big[\eta x\Big]
	\\
	&
	+6\Big[{-}\frac{x}{\xi}\Big]
	-6\Big[{-}\frac{xy}{\eta}\Big]
	+6\Big[x\Big]
	-3\Big[\frac{x y}{\eta\xi}\Big]
	-3\Big[xy\Big].
	\end{aligned}
	\end{align*}	
Then we have
\begin{equation*}
F(x,y):=H(x,y)+H(y,x) \qquad{\text{is mapped to }0\text{ under }} \tD_4\,.
\end{equation*}

\begin{proof}
	In order to verify that \( \tD_4 \big(F(x,y) \big) = 0 \) for all \( x,y \), we apply \cite[Proposition 1]{Zag1990a}, which states that if \( \{n_i, x_i(t)\} \) is a collection of integers \( n_i \) and rational functions of one variable \( x_i(t) \), satisfying
	\begin{equation}\label{eqn:zag_prop1_tensor}
		\sum_i n_i [ x_i(t)]^{m-2} \otimes \big( [x_i(t) ] \wedge [1 - x_i(t)] \big) = 0 \,,
	\end{equation}
	in \( \operatorname{Sym}^{m-2} (\mathbb{C}(t)^\times) \otimes \wedge^2 ( \mathbb{C}(t)^\times) \otimes_\mathbb{Z} \mathbb{Q} \).  Then \( \sum_i n_i \tD_m\big(x_i(t)\big) = \text{constant} \).  In this tensor condition the tensors are multiplicative \( (a b) \otimes c = a \otimes c + b \otimes c \), and we can ignore torsion (multiplication by roots of unity) in each slot.  This tensor condition is closely related to the $\otimes^m$-invariant (``symbol'') of multiple polylogarithms \cite{Go-Gal}, and amounts to a convenient reformulation of the derivative of \( \tD_m(x_i(t)) \) for the purposes of calculation.
	\medskip
	
	Set \( m = 4 \), and fix \( y = y_0 \in \mathbb{C} \), it is then straightforward (if tedious) to check that \eqref{eqn:zag_prop1_tensor} vanishes for the list of coefficients and arguments in \( F(x,y_0) \).  Hence for any fixed \( y = y_0 \) the combination \( \tD_4 (F(x,y_0))  \) is constant.  By the symmetry of \( F(x,y) \) with respect to \( x \leftrightarrow y \), we also have by the same calculation that for any fixed \( x = x_0 \), the combination \(\tD_4 ( F(x_0,y)) \) is constant.  It follows that \( \tD_4 (F(x_1,y_1)) = \tD_4 (F(x_2,y_1)) = \tD_4 (F(x_2,y_2)) \) for any \((x_1,y_1), (x_2,y_2) \in \mathbb{C}^2 \), so \( \tD_4 (F(x,y)) \) is constant overall.  Since \( \tD_4 \) vanishes on the real line, and by specializing for example \( x = y = \tfrac{1}{2} \) all arguments in \( F(x,y) \) are real, this constant is necessarily 0.  We have therefore established the required functional equation.
\end{proof}
\end{lem}

\emph{Step 5}. Specialization to the 12-th roots of unity. In the rest of this section, we put
\begin{equation*}
\rho:=e^{2\pi\mi/12}.
\end{equation*}
Note that by applying \eqref{D4-1} and \eqref{D4-2} at most twice, we can make the argument of $\tD_4$ lie in the upper half unit disk.
We will often apply this rule in our calculations below.

By Lemma \ref{lem-4-1}, we have
\begin{align}
\Ls_4\left(\frac{\pi}{3}\right)=\frac{1}{2}\pi\zeta(3)+\frac{9}{2}\tD_4(\rho^2).\label{eqn-Ls4-Pi3}
\end{align}
We remark that in \cite[Equation (83c)]{BS2013}, J.M. Borwein and A. Straub proved that $\Ls_4(\pi/3)=\pi\zeta(3)/2+9\Cl_{4}(\pi/3)/2$, (here $\Cl_{4}$ is the Clausen function) which agrees with \eqref{eqn-Ls4-Pi3} because $\tD_4(\rho^2)=\Im\Li_4(\rho^2)=\Cl_{4}(\pi/3)$.

By Lemma \ref{lem-4-2}, we have
\begin{align}\label{Lsc32-Pi3-1} \Lsc_{3,2}\left(\frac{\pi}{3}\right)=-\frac{1}{12}\pi\zeta(3)-\tD_4(\rho^2)+\frac{1}{2}\tD_4(\rho^4)
+\frac{5}{2}\tD_4\Big(\frac{\rho}{\sqrt{3}}\Big)-2\tD_4\Big(\frac{\rho^3}{\sqrt{3}}\Big).
\end{align}
By Lemma \ref{lem-D4}, we have that $\tD_4$ vanishes on $F(\rho^2,\rho^4)$, which implies
\begin{align}\label{F24}
-9\tD_4(\rho^2)+6\tD_4(\rho^4)-15\tD_4\Big(\frac{\rho}{\sqrt{3}}\Big)+12\tD_4\Big(\frac{\rho^3}{\sqrt{3}}\Big)=0.
\end{align}
Note the distribution relation
\[\tD_4(\rho^4) \,=\, 8\tD_4(\rho^2)+8\tD_4(-\rho^2) \,=\, 8\tD_{4}(\rho^2)-8\tD_4(\rho^4)\]
implies that
\begin{align}\label{D4-rho4}
\tD_{4}(\rho^{4})=\frac{8}{9}\tD_{4}(\rho^{2}).
\end{align}
Combining \eqref{Lsc32-Pi3-1}, \eqref{F24} and \eqref{D4-rho4}, we obtain
\begin{align}
\Lsc_{3,2}\left(\frac{\pi}{3}\right)=-\frac{1}{12}\pi\zeta(3)-\frac{7}{6}\tD_4(\rho^2).\label{eqn-Lsc32-Pi3}
\end{align}

\bigskip

Write $\rho^{1/2}:=e^{2\pi \mi/24}$ and
\[ r:=|1-\rho|=\frac{\sqrt{6}-\sqrt{2}}{2}. \]
By Lemma \ref{lem-4-1} and Lemma \ref{lem-4-2}, we have
\begin{align}
\Ls_{4}\Big(\frac{5\pi}{6}\Big) & {} =\frac{5}{4}\pi\zeta(3)-\frac{3}{2}\tD_4(\rho^5)+6\tD_4(r\rho^{1/2}),\label{eqn-Ls4-5Pi6}\\
\Lsc_{3,2}\Big(\frac{\pi}{6}\Big) &{} = -\frac{1}{24}\pi\zeta(3)-\tD_4(\rho)+\frac{1}{2}\tD_4(\rho^2)+\frac{1}{2}\tD_4(\rho^5)+2\tD_4(r\rho^{1/2})-2\tD_4(r^2\rho^3),\label{eqn-Lsc32-Pi6}\\
\Lsc_{3,2}\Big(\frac{5\pi}{6}\Big) & {} =-\frac{5}{24}\pi\zeta(3)+\frac{1}{2}\tD_4(\rho)-\frac{1}{2}\tD_4(\rho^2)-\tD_4(\rho^5)+2\tD_4(r\rho^{5/2})-2\tD_4(r^2\rho^3).\label{eqn-Lsc32-5Pi6}
\end{align}

By substituting first \eqref{eqn-4-1}, \eqref{eqn-4-3}, then \eqref{eqn-Ls4-Pi3}, \eqref{eqn-Lsc32-Pi3}, \eqref{eqn-Ls4-5Pi6}, \eqref{eqn-Lsc32-Pi6} and \eqref{eqn-Lsc32-5Pi6}, the left-hand side of \eqref{conj2} is transformed as follows:
\begin{align*}
&41\sum_{n=0}^\infty\frac{\binom{2n}{n}}{(2n+1)^{4}16^n}+9\sum_{n=1}^\infty\frac{\binom{2n}{n}H_{2n}}{(2n+1)^{3}16^n}\\[1ex]
& {} = \frac{41}{6}\Ls_4\Big(\frac{\pi}{3}\Big)+9\Lsc_{3,2}\Big(\frac{\pi}{3}\Big)+18\Ls_{4}\Big(\frac{5\pi}{6}\Big)-18\Lsc_{3,2}\Big(\frac{\pi}{6}\Big)+36\Lsc_{3,2}\Big(\frac{5\pi}{6}\Big)-18\pi\zeta(3)\\[1ex]
& {} = \frac{5}{12}\pi\zeta(3)+36\tD_4(\rho)-\frac{27}{4}\tD_4(\rho^2)-72\tD_4(\rho^5)+72\tD_4(r\rho^{1/2})+72\tD_4(r\rho^{5/2})-36\tD_4(r^2\rho^3).
\end{align*}
Since $\beta(4)=\Im(\Li_4(\mi))=\tD_4(\rho^3)$, Conjecture \eqref{conj2} is reduced to
\begin{align}
\tD_4(\rho)-\frac{3}{16}\tD_4(\rho^2)-\frac{10}{9}\tD_4(\rho^3)-2\tD_4(\rho^5)+2\tD_4(r\rho^{1/2})+2\tD_4(r\rho^{5/2})-\tD_4(r^2\rho^3) \stackrel{?}{=}0. \label{TBP-0}
\end{align}
By the distribution relation
\[\tD_4(r^2\rho^3)=8\tD_4(r\rho^{3/2})+8\tD_4(-r\rho^{3/2})=8\tD_4(r\rho^{3/2})-8\tD_4(r\rho^{9/2}),\]
it remains to show that
\begin{equation} \begin{aligned}
& 2\tD_4(r\rho^{1/2})-8\tD_4(r\rho^{3/2})+2\tD_4(r\rho^{5/2})+8\tD_4(r\rho^{9/2}) \\
& {} \stackrel{?}{=} -\tD_4(\rho)+
\frac{3}{16}\tD_4(\rho^2)+\frac{10}{9}\tD_4(\rho^3)+2\tD_4(\rho^5).
\end{aligned}
 \label{TBP-1}
 \end{equation}

By specializing Lemma \ref{lem-D4} to various choices of \( x, y \) we obtain further relations between \( \tD_4 \).  In particular, since $\tD_4$ vanishes on both \( \frac{1}{3}F(\rho^2,\rho)\) and \(\frac{1}{3}F(\rho^2,\rho^5) \), we have respectively
\begin{align}
5\tD_4(\rho)-3\tD_4(\rho^3)+3\tD_4(r\rho^{1/2})-\tD_4(r\rho^{3/2})-2\tD_4(r\rho^{5/2})+3\tD_4(r\rho^{9/2}) & {} = 0,\label{F21}\\
-3\tD_4(\rho^3)+5\tD_4(\rho^5)-2\tD_4(r\rho^{1/2})-3\tD_4(r\rho^{3/2})+3\tD_4(r\rho^{5/2})+\tD_4(r\rho^{9/2}) & {} =0.\label{F25}
\end{align}
By adding \eqref{F21} and \eqref{F25}, we have
\[\tD_4(r\rho^{1/2})-4\tD_4(r\rho^{3/2})+\tD_4(r\rho^{5/2})+4\tD_4(r\rho^{9/2})=-5\tD_4(\rho)+6\tD_4(\rho^3)-5\tD_4(\rho^5).\]
Therefore, \eqref{TBP-1} is reduced to
\begin{align}
-9\tD_4(\rho)-\frac{3}{16}\tD_4(\rho^2)+\frac{98}{9}\tD_4(\rho^3)-12\tD_4(\rho^5)\stackrel{?}{=}0.\label{TBP-2}
\end{align}
By the distribution relations, we have
\begin{align}
\tD_4(\rho^2)=8\tD_4(\rho)+8\tD_4(-\rho)& \, \quad \Longrightarrow \quad
 \tD_4(\rho^2)=8\tD_4(\rho)-8\tD_4(\rho^5),\label{rho2}\\
\tD_4(\rho^3)=27\tD_4(\rho)+27\tD_4(\rho^5)+27\tD_4(\rho^9)& \, \quad  \Longrightarrow\quad
 \tD_4(\rho^3)=\frac{27}{28}\tD_4(\rho)+\frac{27}{28}\tD_4(\rho^5).\label{rho3}
\end{align}
The equations \eqref{rho2} and \eqref{rho3} establish \eqref{TBP-2}. Therefore the proof of \eqref{conj2} is complete.

\section{Proof of \protect\eqref{conj1}}

Let
\[\phi=\frac{\sqrt{5}+1}{2}\]
be the golden ratio. Setting $p=2$ and $z=\frac{1}{4}$ in \eqref{eqn-lshch-1} and \eqref{eqn-lshch-2}, we have $\gth=\log\phi$ and
\begin{align}
&\sum_{n=0}^\infty \frac{\binom{2n}{n}}{(2n+1)^3(-16)^n}=-\frac{1}{2}\Lsh_3(2\log\phi),\label{eqn-5-1}\\
&\sum_{n=1}^\infty\frac{\binom{2n}{n}H_{2n}}{(2n+1)^2(-16)^n}=-2\Lshch_{2,2}(2\log\phi)+4\Lshch_{1,3}(\log\phi)+4\Lshch_{2,2}(\log\phi).\label{eqn-5-2}
\end{align}

\bigskip

\begin{lem}\label{lem-5-1}
The following expressions for $\Lsh_{3}(x)$, $\Lshch_{1,3}(x)$ and $\Lshch_{2,2}(x)$ hold for all $x \in (0,+\infty)$:
\begin{align}
\Lsh_3(x) ={} & -\tD_3\big(e^{-x}\big) -2\tD_3\big(1-e^{-x}\big) -\frac{1}{3}x\log^2\Big( 2\sinh \frac{x}{2} \Big) + \tD_3(1), \label{eqn-Lsh3}\\
\Lshch_{1,3}(x) ={} & -\tD_3\big({-}e^{-x}\big) -2\tD_3\Big( \frac{1}{1+e^{-x}} \Big) -\frac{1}{3}x\log^2\Big( 2\cosh \frac{x}{2} \Big) + \tD_3(1), \label{eqn-Lshch13}\\
 & \begin{aligned}[c]
\mathllap{ \Lshch_{2,2}(x) = {} }
 & -\frac{1}{8}\tD_3\big(e^{-2x}\big) -\frac{1}{2}\tD_3\big(1-e^{-2x}\big) +\tD_3\big(1-e^{-x}\big) +\tD_3\Big( \frac{1}{1+e^{-x}} \Big) \\
& -\frac{1}{3}x\log\Big( 2\sinh \frac{x}{2} \Big)\log\Big( 2\cosh \frac{x}{2} \Big) -\frac{3}{4}\tD_3(1). \end{aligned} \label{eqn-Lshch22}
\end{align}
\end{lem}

\begin{re}
Using the shorthand $\psi(t) := \tD_3(1-t) - \tD_3(1-1/t)$ and adding suitable 3-term relations for $\tD_3$, Lemma \ref{lem-5-1} can be stated more succinctly and uniformly as follows.
\begin{align*}
 \Lshch_{j,k}(x) \; +\; & \frac{x}{3} \log^{j-1}(2\sinh(x/2))  \log^{k-1}(2\cosh(x/2))  \\
     &= \begin{cases}
     \psi(e^x)  \text{ for }(j,k)=(3,1), \\
     \frac{1}{4}\,( \psi(e^{2x}) - 2\psi(e^x) - 2\psi(e^{-x}) )   \text{ for } (j,k)=(2,2), \\
      \psi(-e^x)   \text{ for } (j,k)=(1,3).
     \end{cases}
 \end{align*}
\end{re}

\begin{proof} Suppose $x\in(0,+\infty)$.
Let $f_1(x)$ be the difference between the left-hand and right-hand sides of \eqref{eqn-Lsh3}. Clearly, $\lim_{x \rightarrow 0} f_1(x)=0$. By the definitions of $\tD_3$ and $\Lsh_3$, and the simple identity $\log(2\sinh(x/2)) = x/2 + \log(1-e^{-x})$, we may rewrite $f_1(x)$ as
\begin{align*}
f_1(x) =& -\int_{0}^{x} \Big( \frac{t}{2} + \log\big(1-e^{-t}\big) \Big)^2 \dd t +\Li_3\big(e^{-x}\big) +x\Li_2\big(e^{-x}\big) +2\Li_3\big(1-e^{-x}\big) \\
&-2\log\big(1-e^{-x}\big)\Li_2\big(1-e^{-x}\big) +x\log^2\big(1-e^{-x}\big) +\frac{1}{12}x^3 -\zeta(3).
\end{align*}
Then a straightforward computation gives $f_1^{\prime}(x) = 0$, which completes the proof of \eqref{eqn-Lsh3}. \medskip

Let $f_2(x)$ be the difference between the left-hand and right-hand sides of \eqref{eqn-Lshch13}. We have $\lim_{x \rightarrow 0} f_2(x) = \tD_3(-1) + 2\tD_3(1/2) -\tD_3(1) =0$ by the following identities:
\begin{align}\label{tD3-neg1-and-tD3-1over2}
\tD_3(-1) = -\frac{3}{4} \tD_3(1) \quad \text{and} \quad \tD_3\Big(\frac{1}{2}\Big) = \frac{7}{8} \tD_3(1).
\end{align}
(The first identity in \eqref{tD3-neg1-and-tD3-1over2} follows from the duplication relation $\tD_3(1)=4\tD_3(1)+4\tD_3(-1)$. See \cite[(6.12) and (1.16)]{Lew1981} for the second.)

By the definitions of $\tD_3$ and $\Lshch_{1,3}$, and the identity $\log\big(2\cosh(x/2)\big) = x/2 + \log\big(1+e^{-x}\big)$, we may rewrite $f_2(x)$ as
\begin{align*}
f_2(x) =& -\int_{0}^{x} \Big( \frac{t}{2} + \log\big(1+e^{-t}\big) \Big)^2 \dd t +\Li_3\big({-}e^{-x}\big) +x\Li_2\big({-}e^{-x}\big) + 2\Li_3\Big( \frac{1}{1+e^{-x}} \Big) \\
& +2\log\big(1+e^{-x}\big)\Li_2\Big( \frac{1}{1+e^{-x}} \Big) +\frac{2}{3}\log^3\big(1+e^{-x}\big) +x\log^2\big(1+e^{-x}\big)
+\frac{1}{12}x^3 -\zeta(3).
\end{align*}
Then a straightforward computation gives $f_2^{\prime}(x) = 0$, which completes the proof of \eqref{eqn-Lshch13}.

Observing that $\log(2\sinh x) = \log\big(2\sinh(x/2)\big) + \log\big(2\cosh(x/2)\big)$, we have
\begin{align*}
\Lsh_3(2x) ={} & {-}2\int_{0}^{x} \log^2(\sinh t) \dd t \\
={} & {-}2\int_{0}^{x} \Big( \log\Big(\sinh\frac{t}{2}\Big) + \log\Big(\cosh\frac{t}{2}\Big) \Big)^2 \dd t \\
={} & 2\Lsh_3(x) + 4\Lshch_{2,2}(x) + 2\Lshch_{1,3}(x).
\end{align*}
Therefore,
\begin{align}
\Lshch_{2,2}(x) = \frac{1}{4}\Lsh_3(2x) -\frac{1}{2}\Lsh_3(x) -\frac{1}{2}\Lshch_{1,3}(x). \label{Lshch22-to-other-Lshch}
\end{align}
\medskip

Equation \eqref{eqn-Lshch22} follows immediately by substituting \eqref{eqn-Lsh3} and \eqref{eqn-Lshch13} into \eqref{Lshch22-to-other-Lshch}, and using the duplication relation $\tD_3\big(e^{-2x}\big)=4\tD_3\big(e^{-x}\big)+4\tD_3\big({-}e^{-x}\big)$.
\end{proof}

\bigskip

Specializing Lemma \ref{lem-5-1} at $x=\log\phi$ and $x=2\log\phi$ (and simplifying the golden ratio combinations via \( 1-\phi^{-2} = \phi^{-1}, 1+\phi^{-2} = \sqrt{5} \phi^{-1}, 1 - \phi^{-4} = \sqrt{5} \phi^{-2} \)) we directly find
\begin{align*}
\Lsh_3(2\log\phi) = {} & {-}\tD_3\Big(\frac{1}{\phi^2}\Big) -2\tD_3\Big(\frac{1}{\phi}\Big) +\tD_3(1),\\
\Lshch_{2,2}(2\log\phi) = {} & {-}\frac{1}{8}\tD_3\Big(\frac{1}{\phi^4}\Big) -\frac{1}{2}\tD_3\Big(\frac{\sqrt{5}}{\phi^2}\Big) +\tD_3\Big(\frac{1}{\phi}\Big) +\tD_3\Big(\frac{\phi}{\sqrt{5}}\Big) -\frac{3}{4}\tD_3(1), \\
\Lshch_{1,3}(\log\phi) = {} & {-}\tD_3\Big({-}\frac{1}{\phi}\Big) -2\tD_3\Big(\frac{1}{\phi}\Big) -\frac{3}{4}\log^3\phi +\tD_3(1), \\
\Lshch_{2,2}(\log\phi) = {} & \frac{7}{8}\tD_3\Big(\frac{1}{\phi^2}\Big) +\frac{1}{2}\tD_3\Big(\frac{1}{\phi}\Big) +\frac{3}{4}\log^3\phi -\frac{3}{4}\tD_3(1).
\end{align*}
Then, by substituting the duplication relation
\[
\tD_3\Big({-}\frac{1}{\phi}\Big) = -\tD_3\Big(\frac{1}{\phi}\Big) +\frac{1}{4}\tD_3\Big(\frac{1}{\phi^2}\Big)
\]
and the following evaluation of $\tD_3(\phi^{-2})$ (see \cite[(6.13) and (1.20)]{Lew1981})
\[
\tD_3\Big(\frac{1}{\phi^2}\Big) = \frac{4}{5}\tD_3(1)
\]
into the above equations, we obtain
\begin{align}
\Lsh_3(2\log\phi) = {} & {-}2\tD_{3}\Big(\frac{1}{\phi}\Big) +\frac{1}{5}\tD_3(1), \label{eqn-Lsh3-2logphi-out} \\
\Lshch_{2,2}(2\log\phi) = {} & {-}\frac{1}{8}\tD_3\Big(\frac{1}{\phi^4}\Big) -\frac{1}{2}\tD_3\Big(\frac{\sqrt{5}}{\phi^2}\Big) +\tD_3\Big(\frac{1}{\phi}\Big) +\tD_3\Big(\frac{\phi}{\sqrt{5}}\Big) -\frac{3}{4}\tD_3(1), \label{eqn-Lshch22-2logphi-out} \\
\Lshch_{1,3}(\log\phi) = {}  & {-}\tD_3\Big(\frac{1}{\phi}\Big) -\frac{3}{4}\log^3\phi +\frac{4}{5}\tD_3(1), \label{eqn-Lshch13-logphi-out}\\
\Lshch_{2,2}(\log\phi) = {}  & \frac{1}{2}\tD_3\Big(\frac{1}{\phi}\Big) +\frac{3}{4}\log^3\phi -\frac{1}{20}\tD_3(1). \label{eqn-Lshch22-logphi-out}
\end{align}

\bigskip

By substituting first \eqref{eqn-5-1}, \eqref{eqn-5-2}, then \eqref{eqn-Lsh3-2logphi-out}--\eqref{eqn-Lshch22-logphi-out}, the left-hand side of \eqref{conj1} is transformed as follows:
\begin{align*}
& 17 \sum_{n=0}^\infty \frac{\binom{2n}{n}}{(2n+1)^3(-16)^n} + 5 \sum_{n=1}^\infty\frac{\binom{2n}{n}H_{2n}}{(2n+1)^2(-16)^n} \\
& {} =  {-}\frac{17}{2} \Lsh_3(2\log\phi) -10 \Lshch_{2,2}(2\log\phi) +20 \Lshch_{1,3}(\log\phi) +20 \Lshch_{2,2}(\log\phi) \\
& {} = \frac{5}{4} \tD_3\Big(\frac{1}{\phi^4}\Big) +5 \tD_3\Big(\frac{\sqrt{5}}{\phi^2}\Big) -3 \tD_{3}\Big(\frac{1}{\phi}\Big) -10 \tD_3\Big(\frac{\phi}{\sqrt{5}}\Big) +\frac{104}{5} \tD_3(1).
\end{align*}
Since $\tD_3(1) = \zeta(3)$ and the right-hand side of \eqref{conj1} is $14\zeta(3)$, Conjecture \eqref{conj1} is equivalent to
\begin{equation}\label{eqn:d3-id-step2}
	\frac{5}{4} \tD_3\Big(\frac{1}{\phi ^4}\Big)
	+5 \tD_3\Big(\frac{\sqrt{5}}{\phi ^2}\Big)
	-3 \tD_3\Big(\frac{1}{\phi }\Big)
	-10 \tD_3\Big(\frac{\phi }{\sqrt{5}}\Big)
	+\frac{34}{5} \tD_3(1) \overset{?}{=} 0.
\end{equation}
We shall prove this by specializing a suitable \( \tD_3 \) functional equation.

\begin{lem}\label{lem:d3fe}
	The following linear combination $G(x)$ vanishes identically under $\tD_3$, where {\small
	\begin{align*}
		G(x) & {} := \begin{aligned}[t]
		&
		5 \Big[\frac{1-2 x}{(1-x)^3 (1+x)}\Big]
		+6 \Big[{-}\frac{(1-x)^3}{(2-x)^3}\Big]
		-6 \Big[\frac{1}{(1-x)^3}\Big]
		-15 \Big[\frac{(1-x) (1+x)}{1-2 x}\Big]
		-15 \Big[\frac{1-2 x}{(1-x)^2}\Big]
			\\
		&
		-18 \Big[\frac{1-x}{(2-x)^2}\Big]
		+18 \Big[{-}\frac{1}{(1-x) (2-x)}\Big]
		-3 \Big[\frac{1}{(1-x) (1+x)}\Big]
		-10 \Big[\frac{1-2 x}{2-x}\Big]
		-10 \Big[\frac{2-x}{1+x}\Big]
		\\
		&
		+15 \Big[{-}\frac{x}{1-2 x}\Big]
		+15 \Big[\frac{x}{1-x}\Big]
		-24 \Big[{-}\frac{1-x}{1+x}\Big]
		+24 \Big[\frac{1-x}{1+x}\Big]
		+45 \Big[\frac{1-2 x}{1-x}\Big]
		-54 \Big[{-}\frac{1-x}{2-x}\Big]
		 \\
		 &
		 +36 \Big[\frac{1}{2-x}\Big]
		 +6 \Big[\frac{1}{1-x}\Big]
		-18 \Big[\frac{1}{1+x}\Big]
		+42 \Big[{-}\frac{1}{1-x}\Big]
		-34 \Big[ 1 \Big]\,.
		\end{aligned}
	\end{align*}
	}
	
	\begin{proof}
		The proof strategy is exactly the same as for Lemma \ref{lem-D4}; we apply the tensor criterion in \eqref{eqn:zag_prop1_tensor} in the case \( m = 3 \).  This shows that \( \tD_3(G(x))\) is constant.  To fix the constant, we specialize to \( x = 0 \).  We find (simplifying only with inversion at the moment) that
		\[
			\tD_3(G(0)) = \begin{aligned}[t]
			&
			18 \tD_3(-1)
			-36 \tD_3\Big({-}\frac{1}{2}\Big)
			+6 \tD_3\Big({-}\frac{1}{8}\Big)
			+30 \tD_3(0) \\
			& -18 \tD_3\Big(\frac{1}{4}\Big)
			+16 \tD_3\Big(\frac{1}{2}\Big)
			-11 \tD_3(1) \,.
			\end{aligned}
		\]
		This time, using the duplication relation \( \tD_3(\tfrac{1}{4}) = 4 \tD_3(\tfrac{1}{2}) + 4 \tD_3(-\tfrac{1}{2}) \) to eliminate \( \tD_3(\tfrac{1}{4}) \), and simplifying with \(  \tD_3(0) = 0 \), we obtain
		\[
			\tD_3(G(0))  =
			18 \tD_3(-1)
			-108 \tD_3\Big({-}\frac{1}{2}\Big)
			+6 \tD_3\Big({-}\frac{1}{8}\Big)
			-56 \tD_3\Big(\frac{1}{2}\Big)
			-11 \tD_3(1)
			\,.
		\]
		But we can show  this vanishes by using \eqref{tD3-neg1-and-tD3-1over2} and the well-known identity (see \cite[pg. 179]{Lew1981})
		\begin{align*}
			 \tD_3\Big({-}\frac{1}{8}\Big) - 18 \tD_3\Big({-}\frac{1}{2}\Big) = \frac{49}{4} \tD_3(1).
		\end{align*}
		With this the functional equation in the lemma is now proven.
	\end{proof}
\end{lem}

Now consider \( \tD_3\big(G(-\phi^{-1})\big) \). We first put all real arguments into the interval \( [0,1] \)
by applying the duplication relation and inversion relation
\[
\tD_3(x^2) = 4 \big( \tD(x) + \tD(-x) \big) \,, \quad
\tD_3(x^{-1}) = \tD_3(x) \,.
\]
Then we obtain exactly
\[
0 = \tD_3\big(G(-\phi^{-1})\big) =
-\frac{25}{4} \tD_3\Big(\frac{1}{\phi ^4}\Big)
- 25 \tD_3\Big(\frac{\sqrt{5}}{\phi ^2}\Big)
+ 15 \tD_3\Big(\frac{1}{\phi }\Big)
+ 50 \tD_3\Big(\frac{\phi }{\sqrt{5}}\Big)
- 34 \tD_3(1) \,.
\]
This is \( {-}5 \) times the left-hand side of \eqref{eqn:d3-id-step2}, hence the left-hand side of \eqref{eqn:d3-id-step2} is equal to exactly $0$. The proof of \eqref{conj1} is complete.

\begin{re}
	It should be noted that the functional equation in Lemma \ref{lem:d3fe} has been concocted to give a simple proof of \eqref{eqn:d3-id-step2} in the previous lines.  This functional equation can be broken down into a number of smaller functional equations, with slightly more structured coefficients.  Specifically Lemma \ref{lem:d3fe} is a combination of the following 4 linearly independent functional equations (irreducible within the selected set of arguments),
	{\small
	\begin{align}
	& \begin{aligned}
	  \tD_3\Big( {} {-} & \Big[\frac{1-2 x}{(1-x)^3 (1+x)}\Big]
	  +3 \Big[\frac{1-2 x}{(1-x)^2}\Big]
	 +3 \Big[\frac{(1-x) (1+x)}{1-2 x}\Big]
	 +3 \Big[\frac{1}{(1-x) (1+x)}\Big]
	 \\
	 &
	 -6 \Big[\frac{1-2 x}{1-x}\Big]
	 +2 \Big[\frac{1-2 x}{2-x}\Big]
	 +2 \Big[\frac{2-x}{1+x}\Big]
	 +6 \Big[{-}\frac{1}{1-x}\Big]
	 -6 \Big[\frac{1}{1+x}\Big]
	 +5\Big[1\Big] \Big) = 0 \,, \end{aligned} \notag \\[1ex]
	 & \begin{aligned}[t]
	  \tD_3\Big(  {-} & \Big[{-}\frac{(1-x)^3}{(2-x)^3}\Big]
	 	 +\Big[\frac{1}{(1-x)^3}\Big]
	   +3 \Big[\frac{1-x}{(2-x)^2}\Big]
	   -3 \Big[{-}\frac{1}{(1-x) (2-x)}\Big]
	   	 +9 \Big[{-}\frac{1-x}{2-x}\Big]
	 \\
	 &	 -12 \Big[{-}\frac{1}{1-x}\Big]
	 -9 \Big[\frac{1}{1-x}\Big]
	  -6 \Big[\frac{1}{2-x}\Big]
	 +6 \Big[1\Big] \Big) = 0 \,, \end{aligned} \notag \\[1ex]
	& \tD_3\Big(2\Big[\frac{1}{(1-x) (1+x)}\Big]-4 \Big[{-}\frac{1-x}{1+x}\Big]+4 \Big[\frac{1-x}{1+x}\Big] -8 \Big[\frac{1}{1-x}\Big]-8 \Big[\frac{1}{1+x}\Big] + 7\Big[1\Big] \Big)  = 0  \,, \notag \\[1ex]
	& \tD_3\Big( \Big[\frac{x}{1-x}\Big] + \Big[\frac{1-2 x}{1-x}\Big]+\Big[{-}\frac{x}{1-2 x}\Big]-\Big[1\Big] \Big) = 0 \label{eqn:threeterm} \,.
	\end{align}
	}
	Each of these can be proven in exactly the same way as Lemma \ref{lem:d3fe} itself.  In fact, the last one \eqref{eqn:threeterm} is (up to inversion) a re-parameterization of the 3-term \cite[Equation (6.10)]{Lew1981} functional equation \( \tD_3(x) + \tD_3(1-x) + \tD_3(1-x^{-1}) = \tD_3(1) \), with \( x \mapsto \frac{x}{1-x} \).
\end{re}

\begin{re}
We originally discovered the proof of \eqref{conj1} by expressing \eqref{eqn-5-1} and \eqref{eqn-5-2} in terms of colored multiple zeta values by applying Au's mechanism developed in \cite{Au2020}. Then \eqref{conj1} follows from the computer-aided proof using Au's Mathematica package. For the detailed definition and introduction of colored multiple zeta values, see \cite[Chapter 13-14]{Zhao2016}.
\end{re}

\bigskip

\noindent{\bf Acknowledgments.}  Steven Charlton is supported by Deutsche Forschungsgemeinschaft Eigene Stelle grant CH 2561/1-1, for Projektnummer 442093436.  SC and Herbert Gangl would like to thank the Isaac Newton Institute for Mathematical Sciences, Cambridge, for support and hospitality during the programme \emph{$K$-theory, algebraic cycles and motivic homotopy theory} where work on this paper was undertaken. This work was supported by EPSRC grant no EP/K032208/1.   Ce Xu is supported by the National Natural Science Foundation of China (Grant No. 12101008), the Natural Science Foundation of Anhui Province (Grant No. 2108085QA01) and the University Natural Science Research Project of Anhui Province (Grant No. KJ2020A0057). Jianqiang Zhao is supported by the Jacobs Prize from The Bishop's School.

\end{document}